\documentclass[twoside,leqno,
]{factaMI}
\usepackage{latexsym,amssymb,amsfonts,amsmath,euscript}
\usepackage[section]{placeins}
\usepackage[dvips]{graphicx}
\input epsf

 \factanis{}   
 \Received{}     

 \AMSMSC{Primary 53B05; Secondary 53C15, 53C21}

\title{EINSTEIN CONNECTION OF A WEAK ALMOST CONTACT METRIC MANIFOLD
\footnote{
This work was supported in part for Prof. Milan Zlatanovi\' c
by the Ministry of Education, Science and Technological Development of the Republic of Serbia (contract reg. no. 451-03-34/2026-03/200124).  Miroslav Maksimovi\'c is supported by the Bulgarian Ministry of Education and  Science,
(contract DO1-67/05.05.2022) and by the Ministry of Education, Science and Technological Deve\-lopment of the Republic of Serbia (contract reg. no. 451-03-34/2026-03/200123).
}}

\author{Vladimir Rovenski, Milan Zlatanovi\'c and Miroslav D. Maksimovi\'c}

\begin{document}

\maketitle

\abstract
Advances in modern physics since Einstein have made the nonsymmetric metric (0,2)-tensor $G=g+F$, where $g$ is a pseudo-Riemannian metric associated with gravity, and $F\ne0$ is a skew-symmetric tensor associated with electromagnetism,
more attractive than~ever.
Einstein considered a linear connection $\nabla$ with torsion $T$
such that $(\nabla_X\,G)(Y,Z)=G(T(Y,X),Z)$.
In this paper, we explicitly present the Einstein connection
of $G=g+F$
using a weak almost contact structure
$(f,\xi,\eta)$ with $g(X,fY)=F(X,Y)$ with a natural~condition (trivial in the almost contact case).
We discuss special Einstein connections, and give an example in terms of the weighted product of almost Hermitian~manifold and a real line.

\vskip1mm\noindent
{\bf Keywords}:
nonsymmetric pseudo-Riemannian manifold,
weak almost contact metric structure,
Einstein connection,
$Q$-$T$-condition.

\endabstract

\section{Introduction}\label{sec:intro}
 \vglue-10pt
 \indent

Einstein worked on various variants of unified field theory,
recently known as non-symmetric gravitational theory (NGT), see~\cite{Ein}.
This theory considers a differen\-tiable manifold $(M,G=g+F)$ equipped with a linear connection $\nabla$
(called here an \textit{Einstein connection}) with torsion $T(X,Y)=\nabla_XY-\nabla_YX-[X,Y]$ satisfying 
\begin{equation}\label{metein0}
 (\nabla_X G)(Y,Z)=-G(T(X,Y),Z) .
\end{equation}
The symmetric part $g$ of the (0,2)-tensor~$G$ is associated with gravity, and the skew-symmetric part $F$ is associated with electromagnetism.
Advances in modern physics since Einstein's time have made the asymmetric metric tensor more attractive than ever, see \cite{Ham-2019,JP-2007,Moffat-95}.
Anticommuting variables, noncommutative geometry and superspace can be related to $F$ -- the antisymmetric part of~$G$. The NGT yields interesting results for the nonsymmetric energy momentum tensor and for dark energy and dark matter.

Recent approaches to modified gravity often rely on differen\-tial geometry, inclu\-ding torsion and non-metricity, as a natural extension of General Relativity. Connections with {totally skew-symmetric torsion} the 3-form $T(X,Y,Z):=g(T(X,Y),Z)$
are important due to the relations with mathematical physics (supersymmetric string theories, non-linear $\sigma$-models, and gravitational mo\-dels), as they admit a coupling to spinor fields and lead to holonomy and rigidity phenomena.
Ivanov and Zlatanovi\'c (see~\cite{IZ1}) presented conditions for the existence and uniqueness of the Einstein connection with totally skew-symmetric torsion on a manifold $(M,G=g+F)$ and gave its explicit expression using an almost contact metric (a.c.m.) structure $(f,\xi,\eta, g)$ with $g(X,fY)=F(X,Y)$.

\smallskip

Weak metric structures by Rovenski and Wolak \cite{rov-survey24} generalize the almost Hermitian and a.c.m. structures, as well as Yano's $f$-structure, and are completely consistent with the study of $G=g+F$ in the NGT with an arbitrary skew-symmetric tensor~$F$.
Rovenski and Zlatanovi\'c~\cite{RZ-1,rst-68,rst-63} were the first to apply the weak almost Hermitian and weak a.c.m. structures to NGT of totally skew-symmetric torsion, and obtained the explicit form of the Einstein connection of a weak almost Hermitian manifold, which extend a result by Prvanovi\'c~\cite{Prvanovic-95}.

\smallskip

In this paper, continuing our study \cite{rst-68}, we explicitly present the Einstein connection of an NGT space with degenerate $F$ modeled by a weak a.c.m. structure, satisfying the $Q$-$T$-condition \eqref{E-Q-T-condition}, which disappears for the  case of an a.c.m. structure.
We~also  explicitly present the special Einstein connection, defined by the additional condition
\eqref{E-cond-2K},
of a weak a.c.m. manifold satisfying the $Q$-$T$-condition.
We~expressed the torsion in terms of $\nabla^g F$ and $dF$, and explained how the weak a.c.m. case differs from the a.c.m. case.
The~presented identities with the torsion of an Einstein connection
provide a new tool for constructing examples and for studying classes of nonsymmetric pseudo-Rieman\-nian manifolds, in particular, with
Gray-Hervella and Chinea-Gonzalez classifications, in Einstein's NGT.

The paper has three sections.
In~Section~\ref{sec:geom}, after the introductory  Section~\ref{sec:intro}, we review basics of NGT,
and discuss the relation between the torsion and contorsion of an Einstein connection. In~Section~\ref{sec:w.a.c.m} we explicitly present general and special Einstein connections of a.c.m. manifolds with the $Q$-$T$-condition, since the condition is trivial for a.c.m. manifolds.

\section{Einstein's nonsymmetric geometry}\label{sec:geom}
 \vglue-10pt
 \indent

A nonsymmetric pseudo-Riemannian manifold $(M,G=g+F)$ equipped with an Einstein connection
$\nabla$ different from the Levi-Civita connection $\nabla^g$ is called an NGT space.
The~basic (0,2)-tensor $G$ of an NGT space decomposes into two parts, the symmetric part $g$ (pseudo-Riemannian metric,
$\det g\ne0$) and the skew-symmetric part $F\ne0$ (fundamental 2-form):
\begin{equation*}
 g(X,Y)=\tfrac12\big[G(X,Y)+G(Y,X)\big], \quad F(X,Y)=\tfrac12\big[G(X,Y)-G(Y,X)\big].
\end{equation*}
The (0,2)-tensor $F$ may have any (not necessarily constant) rank, in particular, may be non-degenerate (of maximal rank).
Thus, we obtain a (1,1)-tensor ${f}\ne0$,
\begin{equation}\label{E-def-f}
 g(X, fY) = F(X,Y) .
\end{equation}
Since $F$ is skew-symmetric, the tensor ${f}$ is also skew-symmetric:
 $g({f}X, Y) = -g(X,{f}Y)$.
Note that $f=-A$, where the tensor $A$ is given in \cite{IZ1} by the equality $g(AX, Y) = F(X,Y)$.
The torsion of a linear connection $\nabla$ on $M$ is given by
\[
 T(X,Y)=\nabla_XY-\nabla_YX-[X,Y].
\]
Separating symmetric and skew-symmetric parts of \eqref{metein0},
we express the
derivatives $\nabla g$ and $\nabla F$ in terms of the (0,3)-torsion tensor $T(X,Y,Z):=g(T(X,Y),Z)$:
\begin{eqnarray}\label{ein6}
 2\,(\nabla_X\,g)(Y,Z)
&&= T(Z,X,\,Y+fY) -T(X,Y,\,Z+fZ) ,\\
\label{ein5}
 2\,(\nabla_Z\,F)(X,Y)
&&= - T(Z,X,\,Y+fY) - T(Y,Z,\,X+fX) .
\end{eqnarray}
Recall the co-boundary formula for a 2-form $F$
(without the coefficient 3, unlike~\cite{Blair-survey}):
\begin{eqnarray}\label{E-3.3a}
 dF(X,Y,Z) &&= X(F(Y,Z)) + Y(F(Z,X)) + Z(F(X,Y)) \notag\\
 &&\ -F([X,Y],Z) - F([Z,X],Y) - F([Y,Z],X).
\end{eqnarray}
The equality \eqref{E-3.3a} yields
\begin{eqnarray}\label{E-3.3}
 dF(X,Y,Z) =
 (\nabla^g_X\,F)(Y,Z)+(\nabla^g_Y\,F)(Z,X)+(\nabla^g_Z\,F)(X,Y).
\end{eqnarray}
The Einstein's metricity condition \eqref{metein0} can be written  in the form
\begin{equation}\label{metricityeqq}
 (\nabla_X(g+F))(Y,Z) = -\,T(X,Y,Z)-T(X,Y,fZ).
\end{equation}
Taking the cyclic sum in \eqref{metricityeqq} and applying the equality
\begin{equation*}
\begin{aligned}
dF(X,Y,Z) &= \,T(X,Y,fZ) +\,T(Y,Z,fX) +\,T(Z,X,fY)\\
&+(\nabla_X F)(Y,Z) +(\nabla_Y F)(Z,X) +(\nabla_Z F)(X,Y) ,
\end{aligned}
\end{equation*}
which follows from \eqref{ein5} and \eqref{E-3.3}, we obtain
\begin{equation}\label{ein2
3.2}
\begin{aligned}
&\quad (\nabla_X\,g)(Y,Z) +(\nabla_Y\,g)(Z,X) +(\nabla_Z\,g)(X,Y) \\
&= -\,dF(X,Y,Z) - T(X,Y,Z) - T(Y,Z,X) - T(Z,X,Y).
\end{aligned}
\end{equation} Since the left hand side of (\ref{ein2 3.2}) is symmetric, while the right hand side is
skew-symmetric, we get the following two relations, see~\cite[Eq.~(3.3)]{IZ1}:
\begin{eqnarray*}
&(\nabla_X\,g)(Y,Z) +(\nabla_Y\,g)(Z,X) +(\nabla_Z\,g)(X,Y) = 0,
\end{eqnarray*}
and
\begin{eqnarray}\label{ein8}
 dF(X,Y,Z) = -T(X,Y,Z) -T(Y,Z,X) -T(Z,X,Y).
\end{eqnarray}
The Einstein connection $\nabla$
is represented in \cite[Eq.~(3.7)]{IZ1} using the torsion $T$ as
\begin{equation}\label{genconein}
 g(\nabla_XY,Z) =g(\nabla^g_XY,Z)
 +\tfrac12\big\{T(X,Y,Z) -T(Z,X,{f}Y) +T(Y,Z,{f}X)\big\} ,
\end{equation}
where $\nabla^g$ is the Levi-Civita connection.
The {\it difference tensor} $K$ of connections $\nabla$ and $\nabla^g$ is
 $K:=\nabla-\nabla^g$.
By \eqref{genconein}, the (0,3)-tensor $K(X,Y,Z)=g(K_X Y, Z)$ and the torsion of an Einstein connection $\nabla$
are expressed linearly in terms of each~other:
\begin{eqnarray}\label{E-tordfnew}
 && 2\,K(X,Y,Z)= T(X,Y,Z) -T(Z,X,{f}Y)+T(Y,Z,{f}X),\\
\label{E-tordfnew2}
 && T(X,Y,Z) = K(X,Y,Z) - K(Y,X,Z) \ \Longleftrightarrow\ T(X,Y) = K_X Y - K_Y X .
\end{eqnarray}
Since $\nabla^g\,g=0$, from \eqref{E-def-f}
we obtain the following:
\begin{eqnarray}\label{E-nabla-F-f}
 (\nabla^g F)(Z,X,Y)= (\nabla^g_{Z}\,F)(X,Y)
 = g(X, (\nabla^g_Z\,f)Y).
\end{eqnarray}
When a connection $\nabla$ preserves the metric tensor: $\nabla g=0$,
$K$ is called the \textit{contorsion tensor};
in this case the tensors $K$ and $T$ are related by
\begin{eqnarray}\label{E-tordfnew3}
 2\,K(Y,Z,X) = T(X,Y,Z) + T(Y,Z,X) - T(Z,X,Y).
\end{eqnarray}
Comparing \eqref{E-tordfnew3} and \eqref{E-tordfnew} yields the
following identity for the torsion of an Einstein connection
satisfying $\nabla g=0$:
\begin{eqnarray*}
 T(Z,X,Y) - T(Y,Z,X) = T(Y,Z,{f}X) -T(Z,X,{f}Y) .
\end{eqnarray*}

\begin{lemma}[(\cite{rst-68}).]\label{L-nabla-g}
For an Einstein connection $\nabla$ we get equivalent conditions

(i) $K(X,Y,Z)=-K(X,Z,Y)$,

(ii) $\nabla g=0$,

(iii) $(\nabla_X F)(Y,Z)=-(\nabla_Y F)(X,Z)$,
or $g((\nabla_X f)Z, Y) = -g((\nabla_Y f)Z, X)$.
\end{lemma}

\begin{remark}\rm
For an Einstein connection on $(M,G=g+F)$, the tensor $K(X,Y,Z)$~is not totally skew-symmetric, and it is interesting to study its particular symmetries.
The~skew-symmetry of $K_X$, \ie,
$K(X,Y,Z)=-K(X,Z,Y)$, see Lemma~\ref{L-nabla-g}(i),
by \eqref{E-tordfnew}, reduces~to
\begin{eqnarray}\label{E-cond-E2}
 T(X,Y,Z) -T(Z,X,Y) - T(Z,X,fY)+T(X,Y,fZ)=0 .
\end{eqnarray}
Using \eqref{E-cond-E2} in \eqref{ein5}, we obtain
 $(\nabla_X F)(Y,Z)= -T(X,Y,Z) -T(X,Y,fZ)$.
\end{remark}

The following property $K_XY=-K_YX$ of the difference tensor $K$
(\ie, the skew--symmetry of \(K(X,Y,Z)\) only in the first two arguments) by \eqref{E-tordfnew} reduces~to
\begin{eqnarray}\label{E-cond-2K}
 T(Y,Z,fX) + T(X,Z,fY) = 0,
\end{eqnarray}
and characterizes \emph{special Einstein connections},
\ie, the symmetric part of an Einstein connection $\nabla$ coincides with the Levi-Civita connection $\nabla^g$, see~\cite{Prvanovic-95}.
In this case, \eqref{E-tordfnew2} yields $K_XY=\tfrac12\,T(X,Y)$, \ie, the special Einstein connection has the form $\nabla=\nabla^g +\tfrac12\,T$.
The~condition \eqref{E-cond-2K} doesn't imply the total skew--symmetry of \(K(X,Y,Z)\), which, in addition, requires the condition \eqref{E-cond-E2}.

\begin{proposition}[(see
\cite{rst-68})]
Let
\eqref{E-cond-E2} be true.
Then we get the equa\-lity
\begin{eqnarray*}
 K(X,Y,Z) = T(Z,Y,X) - \tfrac12\,dF(X,Y,Z);
\end{eqnarray*}
hence, the tensor $K$ is totally skew-symmetric if and only if the torsion $T$ is totally skew-symmetric. Moreover, the total skew-symmetry of $T$ implies the
condition 
\begin{eqnarray*}
 T(fX, Y) = T(X,fY) = -f\,T(X,Y).
\end{eqnarray*}
\end{proposition}

\begin{proposition}[(\cite{Prvanovic-95})]
For an Einstein connection $\nabla$ with torsion $T$, we have
\begin{eqnarray}\label{2.7^*}
\notag
 2(\nabla^g_X F)(Y,Z)= && -T(Z,X,Y) -T(X,Y,Z) -T(fZ,X,fY)  \\
 && -T(X,fY,fZ) +T(Y,fZ,fX) +T(fY,Z,fX). 
\end{eqnarray}
\end{proposition}

In \cite{rst-68}, we explicitly presented the Einstein connection of a nonsymmetric pseudo-Riemannian space $(M,G=g+F)$ with non-degenerate $F$, in particular, of a weak almost Hermitian manifold.
Recall that a pseudo-Riemannian manifold $(M^{2n},g)$ endowed with a skew-symmetric (1,1)-tensor $f$ of maxi\-mal rank
($-f^2$ is not ne\-cessarily the identity endomorphism ${I}$ of $TM$) is called a \emph{weak almost Hermitian manifold} \cite{rov-survey24}.
The $f^2$-$T$-\textit{condition}
\begin{eqnarray}\label{E-Q-torsion}
 T(f^2X, Y) = T(X,f^2Y) = f^2 T(X,Y) ,
\end{eqnarray}
can be
written as
 $T(f^2X, Y,Z) = T(X,f^2Y,Z) = T(X,Y,f^2Z)$.
In this case, we introduce a new (1,1)-tensor $\widetilde Q=-I-f^2$, and the condition \eqref{E-Q-torsion}
reads
\begin{eqnarray}\label{E-tildeQ-torsion}
 T(\widetilde QX, Y,Z) = T(X,\widetilde QY,Z) = T(X,Y,\widetilde QZ) .
\end{eqnarray}

The following theorem generalizes  \cite[Theorem 1]{Prvanovic-95};
for a non-degenerate tensor $P={I}-f^2$ we can use it to
determine~$T$, and hence, an Einstein connection~$\nabla$.

\begin{theorem}
Let $(f,g)$ be a weak almost Hermitian structure on
a nonsymmetric pseudo-Riemannian manifold $(M^{2n},G=g+F)$, where $F(X,Y)=g(X,fY)$.
Suppose that an Einstein connection $\nabla$ on $M$ satisfies the $f^2$-$T$-condition \eqref{E-Q-torsion},
and the tensor $P:={I}-f^2$ has rank $2n$.
Then the torsion $T$ of $\nabla$ is given~by
\begin{eqnarray}\label{Eq-QT-solution}
\notag
 && 2\,T(Y,Z,\,(I+\tfrac12\,\widetilde Q)^2f^4 X) =
 (\nabla^g_X F)((f+f^3)Y,fZ)
 +(\nabla^g_Y F)(f^2Z,X) \\
\notag
&&\
+(\nabla^g_Z F)(f^2X,Y)
-(\nabla^g_{fX} F)(f^3Y,Z)
-(\nabla^g_{fX} F)(f^2Y,fZ) \\
\notag
&&\
+(\nabla^g_{fY} F)(f^3Z,X)
+(\nabla^g_{fZ} F)(f^2X,fY)
+(\nabla^g F)(P^{-1}(2\,\widetilde Q+\widetilde Q^2)f^2X,fY,Z)\\
\notag
&&\
-(\nabla^g F)(P^{-1}(2\,\widetilde Q+\widetilde Q^2)f^2X,Y,fZ)
-(\nabla^g F)(\widetilde Qf^2X,Y,Z) \\
&&\qquad
 -\tfrac12\,dF((3\,\widetilde Q+2\,\widetilde Q^2)X,fY,fZ)
 +\tfrac12\,dF((3\,\widetilde Q+\widetilde Q^2) f^2X,Y,Z) .
\end{eqnarray}
\end{theorem}

\begin{corollary}[(see Corollary 4.6 in \cite{rst-68}
and Theorems 1 and 2 of \cite{Prvanovic-95}).]
Let $(J,g)$ be an almost Hermitian structure on
a
manifold $(M^{2n},G=g+F)$, where $F(X,Y)=g(X,JY)$.
Then the torsion of an Einstein connection on $M$ is given by
\begin{eqnarray}\label{Eq-2.13}
\notag
 && 2\,T(Y,Z,X) = 2\,(\nabla^g_{JX} F)(JY,Z)
  - (\nabla^g_{JY} F)(JZ,X) -(\nabla^g_{JZ} F)(X,JY)  \\
 &&\qquad\qquad\quad\ -\,(\nabla^g_Y F)(Z,X) -(\nabla^g_Z F)(X,Y) ;
\end{eqnarray}
in particular, the torsion of a special Einstein connection $\nabla$ on $M$, see \eqref{E-cond-2K}, is given~by
\begin{equation*}
 2\,T(X,Y,Z) = (\nabla^g_XF)(Y,Z)
 -(\nabla^g_{JZ}F)(JX,Y) -(\nabla^g_{JY}F)(JX,Z) .
\end{equation*}
\end{corollary}

Among the sixteen Gray-Hervella classes \cite{gh-1980} of almost Hermitian manifolds, the condition \eqref{E-cond-2K} of a special Einstein connection is satisfied for the following classes:
\[
{\mathcal W}_1,\ \
{\mathcal W}_3,\ \
{\mathcal W}_4,\ \
{\mathcal W}_3\oplus {\mathcal W}_4,\ \
{\mathcal W}_1\oplus {\mathcal W}_3,\ \
{\mathcal W}_1\oplus {\mathcal W}_4,\ \
{\mathcal W}_1\oplus {\mathcal W}_3\oplus {\mathcal W}_4;
\]
consequently, the Einstein connection \eqref{Eq-2.13} of an almost Hermitian manifold, belonging to any of these classes is special, see Theorem 3 of~\cite{Prvanovic-95}.

\section{Einstein connection of  a weak a.c.m. manifold}\label{sec:w.a.c.m}
 \vglue-10pt
 \indent

Here, for a weak a.c.m. structure equipped with a connection with torsion $T$, we introduce a new $Q$-$T$-condition (trivial for the a.c.m. case) and explicitly represent an Einstein connection on the corresponding NGT space satisfying this condition.

\begin{definition}
\rm
A \textit{weak a.c.m. structure} on a connected differentiable manifold $M^{2n+1}$ is a set $(f,Q,\xi$, $\eta,g)$, where
$g$ is a (pseudo-)Riemannian metric on $M$,
$f$ is a
(1,1)-tensor
such that the distribution $\ker\eta$ is $f$-invariant, \ie, $f(\ker\eta)\subset\ker\eta$,
$\xi$~is a unit vector field,
$\eta$ is a 1-form such that ${\eta}({\xi})=1$, satisfying
\begin{eqnarray}\label{2.2}
 & g({f} X,{f} Y)= g(X,Q\,Y)
 -{\eta}(X)\,{\eta}(Y),\quad
\end{eqnarray}
and $Q:=-{f}^2 + \eta\otimes\xi$ is a nonsingular (1,1)-tensor, see \cite{rov-survey24}.
A \textit{weak contact metric structure} is a weak a.c.m. structure satisfying $d\eta(X,Y)=g(X,fY)$.
\end{definition}

For a weak a.c.m. structure, we define $\widetilde Q:=Q-I$ (which is different from the definition of $\widetilde Q$ for a weak almost Hermitian structure).

Note that $Q=I$
for a.c.m. manifolds.
For a weak a.c.m. structure $({f},Q,\xi,\eta,g>0)$, the
tensor $Q$ is self-adjoint and positive definite, and the following hold:
\begin{eqnarray*}
 {f}\,\xi=0,\quad
 \eta\circ{f}=0,\quad
 [Q,\,{f}]=0,\quad
 Q\,\xi=\xi,\quad
 \eta\circ Q=\eta.
\end{eqnarray*}
Therefore,
$\ker f=\{\xi\}$
and ${\rm Im}\,f=\ker\eta$.

For a weak contact metric structure, $\xi$ is a geodesic vector field, that is, $\nabla^g_\xi\,\xi=0$.

\begin{definition}\rm
Let $\nabla$ be a linear
connection with torsion $T$ on a weak a.c.m. manifold $(M^{2n+1},f,Q,\xi,\eta,g)$.
We introduce the $Q$-$T$-\textit{condition} (that disappears when $Q=I$):
\begin{eqnarray}\label{E-Q-T-condition}
 T(QX, Y,Z) = T(X,QY,Z) = T(X,Y,QZ) .
\end{eqnarray}
\end{definition}
Note that \eqref{E-Q-T-condition} is not equivalent to
the $f^2$-$T$-condition \eqref{E-Q-torsion} for a weak a.c.m. manifold, however, we can write \eqref{E-Q-T-condition} as
\eqref{E-tildeQ-torsion}, but with a different tensor $\widetilde Q$.

\begin{lemma}\label{L-QT-df-nablaF}
Let an Einstein connection $\nabla$ with torsion $T$ on a weak a.c.m. mani\-fold
satisfy the $Q$-$T$-condition \eqref{E-Q-T-condition}. Then the following equali\-ties~are~valid:
\begin{eqnarray}\label{E-Q-T-conditiondF}
 && dF(QX, Y,Z) = dF(X,QY,Z) = dF(X,Y,QZ), \\
\label{E-Q-T-conditionF}
 && (\nabla^g_{QX} F)(Y,Z) =(\nabla^g_X F)({QY},Z)=(\nabla^g_X F)(Y,{QZ}).
\end{eqnarray}
\end{lemma}

\begin{proof}
Based on \eqref{E-Q-T-condition} and \eqref{ein8},
it is easy to verify that
\eqref{E-Q-T-conditiondF} is true. From \eqref{E-Q-T-condition} and \eqref{2.7^*},
using commutativity for $f$ and $Q$, we get the equality \eqref{E-Q-T-conditionF}.
\end{proof}

\begin{example}\rm
Let $(M^{2n+1},f,\xi,\eta,g)$ be an a.c.m.
(or, contact metric) manifold. Set
 \begin{eqnarray}\label{E-cont-lambda}
 f' = \lambda\,f,\qquad
 Q= \lambda^2 I + (1-\lambda^2)\,\eta\otimes\xi ,
 \qquad
 g' = \lambda^{-1}\,g+(1-\lambda^{-1})\,\eta \otimes\eta,
 \end{eqnarray}
where $\lambda>0$ is a differentiable function.
Then $(f',Q,\xi,\eta,g')$ is a weak a.c.m. (or, weak contact metric)
structure, \eg, \cite[Proposition~5.1]{rov-survey24}.
This construction shows that some weak a.c.m. manifolds naturally arise from classical ones
by the deformation~\eqref{E-cont-lambda}.

Let $\nabla$ be a connection on $M$ with torsion $T$ and the
torsion be adapted to the decomposition $TM=\ker\eta\oplus\{\xi\}$, that is,
$T(X,Y,Z)=T(P^\top X,P^\top Y,P^\top Z)+T(P^\bot X,P^\bot Y,P^\bot Z)$.
Using $T(QX,Y,Z)=\lambda^2T(X,Y,Z)+(1-\lambda^2)\eta(X)T(\xi,Y,Z)$, and similarly for
$T(X,QY,Z)$ and $T(X,Y,QZ)$, we conclude that the $Q$-$T$-condition \eqref{E-Q-T-condition} is satisfied.
\end{example}


The following lemma is used in the proof of Theorem~\ref{T-w.c.m} in what follows.

\begin{lemma}
Let an NGT space $(M^{2n+1},G=g+F,\nabla)$ admits a weak a.c.m. structure $(f,Q,\xi,\eta,g)$ with $F(X,Y)=g(X,fY)$.
If the $Q$-$T$-condition \eqref{E-Q-T-condition} is true, then
\begin{eqnarray}
\label{E-fy-fz2-xi}
\notag
&& T(fY,fZ,(I+Q)X) = -T(Y, Z,(Q+Q^2)X) +\eta(Z)\,T(Y,\xi,(I+Q)X) \\
\notag
&&\quad +\eta(Y)\,\{T(f^2Z,X,\xi) + T( X,\xi,f^2Z)
 -T(\xi,fZ,fX)\} \\
\notag
&&\quad -2\,(\nabla^g_{X}\,F)(Y,f^2Z)
+2\,(\nabla^g_{X}\,F)(fY,fZ) \\
&&\quad+dF(Y,f^2Z,(I+Q)X) -dF(fY,fZ,(I+Q)X) ,  \\
\label{E-fy-fz3-xi}
\notag
&& T(fY,Z,(I+Q)X) = -T(Y,fZ,\,(I+Q)X) +T(Y,Z,\,(Q^2-I)X) \\
 \notag
&&\quad -\eta(Y)\,\{T(f^2Z,X,\xi) +T(\xi,X,f^2Z) -T(\xi,fZ,fX)\} \\
\notag
&&\quad +\eta(Z)\,T(\xi,Y,(I+Q)X)
 {+}2\,(\nabla^g_{X}\,F)(Y,f^2Z)
 {-}2\,(\nabla^g_{X}\,F)(fY,fZ)  \\
&&\quad+2(\nabla^g_{(I+Q)X}\,F)(Y,Z)
 -dF(Y,\,(I+f^2)Z,\,(I+Q)X)  .
\end{eqnarray}
\end{lemma}

\begin{proof}
Using $Z\to fZ$ in \eqref{2.7^*}
and $Q=-f^2+\eta\otimes\xi$, we get
\begin{eqnarray*}
 & 2(\nabla^g_{X}\,F)(Y,fZ)
 =-T(fZ,X,Y) -T(X,Y,fZ) +T(QZ,X,fY) \\
 &-\eta(Z)\,T(\xi,X,fY)
 +T(X,fY,QZ) -\eta(Z)\,T(X,fY,\xi) \\
 & -T(Y,QZ,fX) +\eta(Z)\,T(Y,\xi,fX) +T(fY,fZ,fX) .
\end{eqnarray*}
Using $Y\to fY$ in \eqref{2.7^*}, we get
\begin{eqnarray*}
 & 2(\nabla^g_{X}\,F)(fY,Z) =-T(Z,X,fY) -T(X,fY,Z) +T(fZ,X,QY) \\
 & -\eta(Y)\,T(fZ,X,\xi) +T(X,QY,fZ) -\eta(Y)\,T(X,\xi,fZ) \\
 & -T(QY,Z,fX) +\eta(Y)\,T(\xi,Z,fX) +T(fY,fZ,fX) .
\end{eqnarray*}
Subtracting the above equations and using the $Q$-$T$-condition \eqref{E-Q-T-condition}, we get
\begin{eqnarray*}
 & \{T(Z,X+QX,fY) +T(X+QX,fY,Z)\} \\
 &\ -\{T(X+QX,Y,fZ) +T(fZ,X+QX,Y)\} \\
 &\ \ +\eta(Z)\,\{T(Y,\xi,fX)-T(\xi,X,fY) -T(X,fY,\xi)\} \\
 &\ \ +\eta(Y)\,\{T(fZ,X,\xi) +T(X,\xi,fZ) -T(\xi,Z,fX)\} \\
 &=2\,(\nabla^g_{X}\,F)(Y,fZ)
  -2\,(\nabla^g_{X}\,F)(fY,Z).
\end{eqnarray*}
Applying \eqref{ein8} twice to the above equation, we get
\begin{eqnarray}\label{E-fy-fz1-xi}
\notag
&& T(fY,Z,X+QX) = T(Y,fZ,X+QX)
-2\,(\nabla^g_{X}\,F)(Y,fZ) \\
\notag
&&\quad
+2\,(\nabla^g_{X}\,F)(fY,Z)
+dF(Y,fZ,X+QX) -dF(fY,Z,X+QX) \\
\notag
&&\quad +\eta(Y)\,\{T(fZ,X,\xi) +T(X,\xi,fZ) -T(\xi,Z,fX)\} \\
&&\quad +\eta(Z)\,\{T(Y,\xi,fX)-T(X,\xi,fY) -T(X,fY,\xi)\} .
\end{eqnarray}
Replacing $Z$ with $fZ$ in \eqref{E-fy-fz1-xi} and using the $Q$-$T$-condition \eqref{E-Q-T-condition}, gives~\eqref{E-fy-fz2-xi}.
Using \eqref{ein8}, we rewrite \eqref{2.7^*} as
\begin{eqnarray*}
 && T(fY,Z,X) = -T(Y,fZ,X) -T(Y,Z,X)-dF(Y,Z,X) \\
 &&\quad-T(fY,fZ,X)-dF(fY,fZ,X) +2\,(\nabla^g_X\,F)(Y,Z).
\end{eqnarray*}
Replacing $X$ with $(I+Q)X$ in the above equation
and using \eqref{E-fy-fz2-xi}, we obtain~\eqref{E-fy-fz3-xi}.
\end{proof}


Assuming that an Einstein connection satisfying the $Q$-$T$-condition is given, we present its torsion tensor explicitly.

\begin{theorem}\label{T-w.c.m}
Let $(f,Q,\xi,\eta,g)$ be a weak a.c.m. structure on an NGT space $(M^{2n+1},G=g+F,\nabla)$ with $F(X,Y)=g(X,fY)$. If the $Q$-$T$-condition \eqref{E-Q-T-condition} is true, then $\nabla^g_{\xi}\,\xi=0$ holds, and for all $X,Y\in\mathfrak{X}(M)$ we have
\begin{eqnarray}\label{E-w.c.m-xi}
 &&\qquad T(\xi,Y,\xi) = T(Y,\xi,\xi) = 0 , \\
\label{E-xi-YX}
 &&\qquad T(\xi,Y,X) = dF((I+f-\widetilde Q)^{-1}Y,\xi,X) -2\,(\nabla^g_X\,F)((I+f-\widetilde Q)^{-1}Y,\xi), \\
\label{E-x-y-xi}
\notag
 &&\qquad T(X,Y,\xi) =  dF((I+f-\widetilde Q)^{-1}Y,\xi,X) -2\,(\nabla^g_X\,F)((I+f-\widetilde Q)^{-1}Y,\xi) \\
 &&\qquad\ -dF((I+f-\widetilde Q)^{-1}X,\xi,Y) +2\,(\nabla^g_Y\,F)((I+f-\widetilde Q)^{-1}X,\xi) -dF(X,Y,\xi).
\end{eqnarray}
Moreover, if the tensor $P:={I}-f^2$ has rank $2n+1$,
then
for $X,Y,Z\in\ker\eta$
we get
\begin{eqnarray}\label{Eq-QT-solution2}
\notag
 && 2\,T(Y,Z,\,(I+\tfrac12\,\widetilde Q)^2f^4 X) =
 (\nabla^g_X F)((f+f^3)Y,fZ)
 +(\nabla^g_Y F)(f^2Z,X) \\
\notag
 &&\ +(\nabla^g_Z F)(f^2X,Y)
 -(\nabla^g_{fX} F)(f^3Y,Z)
-(\nabla^g_{fX} F)(f^2Y,fZ) \\
\notag
&&\ +(\nabla^g_{fY} F)(f^3Z,X) +(\nabla^g_{fZ} F)(f^2X,fY)
  +(\nabla^g F)(P^{-1}(2\,\widetilde Q+\widetilde Q^2)f^2X,fY,Z) \\
\notag
&&\ -(\nabla^g F)(P^{-1}(2\,\widetilde Q+\widetilde Q^2)f^2X,Y,fZ)
 -(\nabla^g F)(\widetilde Qf^2X,Y,Z) \\
&&\quad
-\tfrac12\,dF((3\,\widetilde Q+2\,\widetilde Q^2)X,fY,fZ)
+\tfrac12\,dF((3\,\widetilde Q+\widetilde Q^2)f^2X,Y,Z).
\end{eqnarray}
\end{theorem}

\begin{proof}
Applying $X=Y=\xi$ to \eqref{2.7^*}, we have
the equality
\begin{eqnarray}\label{E-w.c.m0}
 T(Z,\xi,\xi) = 2\,(\nabla^g_\xi F)(Z,\xi).
\end{eqnarray}
From \eqref{E-fy-fz3-xi} with $X=Z=\xi$ we obtain
\begin{eqnarray}\label{E-w.c.m-xiB}
 T(Y+fY,\xi,\xi)  = 2\,(\nabla^g_\xi F)(Y,\xi) .
\end{eqnarray}
Comparing \eqref{E-w.c.m0} and \eqref{E-w.c.m-xiB},
we get \eqref{E-w.c.m-xi}
and
$(\nabla^g_\xi F)(Y,\xi)=0$.
Hence, $(\nabla^g_{\xi}\,f)\xi=0$ is true,
that is, $\xi$ is a geodesic vector field: $\nabla^g_{\xi}\,\xi=0$.
Due~to the skew-symmetry in the first two arguments,
$T(\xi,Y,\xi)=0$ is also true.
Using \eqref{E-Q-T-condition}, \eqref{E-Q-T-conditiondF}, \eqref{E-Q-T-conditionF} and $Z=\xi$ in \eqref{E-fy-fz3-xi}, we obtain \eqref{E-xi-YX}.
Using \eqref{E-xi-YX} in \eqref{ein8} with $Z=\xi$, we obtain \eqref{E-x-y-xi}:
\begin{eqnarray*}
 && T(X,Y,\xi) = -dF(X,Y,\xi) +T(\xi,Y,X) -T(\xi,X,Y) \\
 &&\quad =
  dF((I+f-\widetilde Q)^{-1}Y,\xi,X) - 2\,(\nabla^g_X\,F)((I+f-\widetilde Q)^{-1}Y,\xi) \\
 &&\quad
 -dF((I+f-\widetilde Q)^{-1}X,\xi,Y) + 2\,(\nabla^g_Y\,F)((I+f-\widetilde Q)^{-1}X,\xi) -dF(X,Y,\xi) .
\end{eqnarray*}
On $\ker\eta$
we have $\widetilde Q=-f^2-I$. Thus, the torsion of $\nabla$ for $X,Y,Z\in\ker\eta$ is given by \eqref{Eq-QT-solution2} -- the same formula as \eqref{Eq-QT-solution} for a weak almost Hermitian~manifold.
\end{proof}

\begin{remark}\rm
By Theorem~\ref{T-w.c.m}, the property $\nabla^g_\xi\,\xi=0$ is necessary
for the existence of an Einstein connection on a weak a.c.m. manifold satisfying the $Q$-$T$-condition \eqref{E-Q-T-condition}.
Weak contact metric manifolds (for example, weak K-contact, Sasaki\-an,
and weak $\beta$-Kenmotsu manifolds, given by \eqref{2.2} and $(\nabla^g_X\,f)Y=\beta\{g(fX,Y)\,\xi-\eta(Y)\,fX\}$, see \cite{rov-survey24})
satisfy $\nabla^g_\xi\,\xi=0$, and by Theorem~\ref{T-w.c.m}, admit an Einstein connection.
\end{remark}

For a nonsymmetric Riemannian space $(M,G=g+F)$,
the tensor $P={I}-f^2$ is positive definite, and hence, non-degenerate.
Therefore, we have the following.

\begin{corollary}
Let $\nabla$ be an Einstein connection of a weak a.c.m. manifold
considered as a nonsymmetric Riemannian space 
with $F(X,Y)=g(X,fY)$.
If~the $Q$-$T$-condition \eqref{E-Q-T-condition} is valid,
then
$\nabla^g_{\xi}\,\xi=0$
holds and the torsion $T$ of $\nabla$ is given by
\eqref{E-w.c.m-xi}--\eqref{Eq-QT-solution2}.
\end{corollary}

The following result generalizes Theorems 3.3 and 3.7
in \cite{rst-68}.

\begin{theorem}\label{T-s-w.c.m}
Let an a.c.m. manifold $(M^{2n+1},f,\xi,\eta,g)$ with an Einstein connection $\nabla$ be
considered as a nonsymmetric
manifold $(M,G=g+F)$, where $F(X,Y)=g(X,fY)$.
Then
$\nabla^g_{\xi}\,\xi=0$
and
\eqref{E-w.c.m-xi}--\eqref{E-x-y-xi} are true,
and for $X,Y,Z\in\ker\eta$ we~get
\begin{eqnarray}\label{Eq-2.13b}
\notag
 && 2\,T(Y,Z,X) = 2(\nabla^g_{fX}F)(fY,Z) -(\nabla^g_{fY}F)(fZ,X) -(\nabla^g_{fZ}F)(X,fY) \\
 &&\qquad\qquad\qquad -\,(\nabla^g_YF)(Z,X) -(\nabla^g_ZF)(X,Y).
\end{eqnarray}
For a special Einstein connection $\nabla$ on $M$, see \eqref{E-cond-2K}, for $X,Y,Z\in\ker\eta$ we have
\begin{equation}\label{Eq-2.13-sp}
 T(X,Y,Z) =
 (\nabla^g_{fZ}F)(fX,Y)
+(\nabla^g_{fY}F)(fX,Z)
-(\nabla^g_XF)(Y,Z) .
\end{equation}
\end{theorem}

\begin{proof}
{
Although for $Q=I$ the equation \eqref{Eq-QT-solution2} reduces to \eqref{Eq-2.13b}, we will present its proof separately.}
Let $X,Y,Z\in\ker\eta$, then
\(
\eta(X)=\eta(Y)=\eta(Z)=0,
\)
and
\(
f^2X=-X,\ f^2Y=-Y,\ f^2Z=-Z.
\)
Differentiating the equality
$f^2=-I+\eta\otimes\xi$
and using $Z\in\ker\eta$, gives
\begin{equation}\label{acm1}
(\nabla^g_Xf)\,fZ+f(\nabla^g_Xf)Z
=(\nabla^g_X\eta)(Z)\,\xi .
\end{equation}
Using \eqref{acm1} and \eqref{E-nabla-F-f},
we get for $X,Y,Z\in\ker\eta$:
\begin{equation}\label{acm2}
(\nabla^g_XF)(fY,fZ)=-(\nabla^g_XF)(Y,Z),\quad
(\nabla^g_XF)(fY,Z)=(\nabla^g_XF)(Y,fZ).
\end{equation}

Replacing $Y\mapsto fY$ and $Z\mapsto fZ$
in \eqref{2.7^*}, then using $f^2=-I$ on $\cal D$, yields
\begin{eqnarray}\label{acmM1}
\notag
 && 2(\nabla^g_XF)(fY,fZ) = -T(fZ,X,fY)-T(X,fY,fZ)-T(Z,X,Y) \\
&&\qquad\qquad\qquad\qquad -T(X,Y,Z)-T(fY,Z,fX)-T(Y,fZ,fX).
\end{eqnarray}
Subtracting \eqref{acmM1} from \eqref{2.7^*} and using \eqref{acm2}, we obtain
\begin{equation}\label{acmM2}
2(\nabla^g_XF)(Y,Z)=T(Y,fZ,fX)+T(fY,Z,fX).
\end{equation}
Substituting \eqref{acmM2} into \eqref{2.7^*}, we get
\begin{equation}\label{acmM3}
T(X,Y,Z)+T(Z,X,Y)+T(X,fY,fZ)+T(fZ,X,fY)=0.
\end{equation}
Replacing $X\mapsto fX$, $Y\mapsto fY$ in \eqref{acmM2}, we obtain
\begin{equation}\label{acmM4}
T(fY,fZ,X)=T(Y,Z,X)-2(\nabla^g_{fX}F)(fY,Z).
\end{equation}
The cyclic sum of \eqref{acmM4}, obtained using
\eqref{acmM3} and \eqref{acm2}, is
\begin{eqnarray}\label{acmM4-cycllic}
\notag
 && (\nabla^g_XF)(Y,Z) +(\nabla^g_YF)(Z,X) +(\nabla^g_ZF)(X,Y) \\
 && = -T(X,Y,Z) {-}T(Y,Z,X) {-}T(Z,X,Y).
\end{eqnarray}
Replacing $Y\mapsto fY$ and $Z\mapsto fZ$ in \eqref{acmM4-cycllic}, we get
\begin{eqnarray}\label{acmM5}
\notag
&-(\nabla^g_XF)(Y,Z)
+(\nabla^g_{fY}F)(fZ,X)+(\nabla^g_{fZ}F)(X,fY) \\
&= -T(X,fY,fZ)-T(fY,fZ,X)-T(fZ,X,fY).
\end{eqnarray}
From \eqref{acmM5}, using \eqref{acmM4} and \eqref{acmM3}, we obtain
\begin{eqnarray}\label{acmM7}
\notag
&& -(\nabla^g_XF)(Y,Z)+(\nabla^g_{fY}F)(fZ,X)+(\nabla^g_{fZ}F)(X,fY) \\
&& = T(X,Y,Z)+T(Z,X,Y)-T(Y,Z,X)  +2(\nabla^g_{fX}F)(fY,Z).
\end{eqnarray}
The cyclic sum of \eqref{acmM7} is
\begin{eqnarray*}
 && (\nabla^g_{fY}F)(fZ,X) +(\nabla^g_{fZ}F)(X,fY) +(\nabla^g_YF)(Z,X)+(\nabla^g_ZF)(X,Y) \\
 && = -2\,T(Y,Z,X) +2(\nabla^g_{fX}F)(fY,Z).
\end{eqnarray*}
Therefore, \eqref{Eq-2.13b} is true.

Let's prove \eqref{Eq-2.13-sp}. Since \(\nabla\) is a special Einstein connection, \eqref{E-cond-2K} is true.
Using \eqref{acmM2}, we transform the two terms on the left-hand side by \eqref{E-cond-2K}, and obtain for $X,Y,Z\in\ker\eta$,
\begin{eqnarray*}
&&\hskip-3mm T(Y,fZ,fX) =-T(X,fZ,fY)=T(fZ,X,fY) =-T(Y,X,Z)=-T(X,Y,Z) , \\
&&\hskip-3mm T(fY,Z,fX) = -T(X,Z,f^2Y) = T(X,Z,Y).
\end{eqnarray*}
Hence, \eqref{2.7^*} reduces to the following:
\begin{equation}\label{spA}
2(\nabla^g_XF)(Y,Z)=-T(X,Y,Z)+T(X,Z,Y).
\end{equation}
Applying \eqref{acmM2} to \eqref{spA} with \(X\mapsto fZ\), \(Y\mapsto fX\), \(Z\mapsto Y\), and \(f^2=-I\) on \(\ker\eta\), gives
\begin{eqnarray*}
 2(\nabla^g_{fZ}F)(fX,Y) && =T(fX,fY,f^2Z)+T(f^2X,Y,f^2Z) \\
 && =-T(fX,fY,Z)+T(X,Y,Z).
\end{eqnarray*}
Similarly, applying \eqref{acmM2} to \eqref{spA} with \(X\mapsto fY\), \(Y\mapsto fX\), \(Z\mapsto Z\), we obtain
\begin{eqnarray}\label{spC}
\notag
2(\nabla^g_{fY}F)(fX,Z) &&=T(fX,fZ,f^2Y)+T(f^2X,Z,f^2Y) \\
&&=-T(fX,fZ,Y)+T(X,Z,Y).
\end{eqnarray}
Subtracting \eqref{spC} from \eqref{spA}, we obtain
\begin{eqnarray*}
&&2(\nabla^g_XF)(Y,Z)
-2(\nabla^g_{fZ}F)(fX,Y)
-2(\nabla^g_{fY}F)(fX,Z) \\
&&=
\bigl(-T(X,Y,Z)+T(X,Z,Y)\bigr)
-\bigl(-T(fX,fY,Z)+T(X,Y,Z)\bigr) \\
&&\quad
-\bigl(-T(fX,fZ,Y)+T(X,Z,Y)\bigr) \\
&&=
-2T(X,Y,Z)+T(fX,fY,Z)+T(fX,fZ,Y).
\end{eqnarray*}
 Using \eqref{E-cond-2K}, we have
$T(fY,fX,Z)+T(fZ,fX,Y)=0$;
therefore,
\[
  2(\nabla^g_XF)(Y,Z)
 -2(\nabla^g_{fZ}F)(fX,Y)
 -2(\nabla^g_{fY}F)(fX,Z)
=-2\,T(X,Y,Z),
\]
that is, \eqref{Eq-2.13-sp} is true.
\end{proof}

The Einstein connection expressed using covariant derivative $\nabla^g F$ is obtained based on Theorem~\ref{T-s-w.c.m} and the equations \eqref{genconein}, \eqref{Eq-2.13b}, and \eqref{acm2}. For a special Einstein connection, we use \eqref{E-cond-2K} and \eqref{Eq-2.13-sp}.

\begin{corollary}
The Einstein connection $\nabla$ of an a.c.m. manifold
$(M^{2n+1},f,\xi,\eta,g)$
is given for $X,Y,Z\in\ker\eta$ by
\begin{eqnarray}\label{genconein-acm} \notag
 && g(\nabla_XY,Z) = g(\nabla^g_XY,Z)
 + \tfrac{1}{4} \big\{ 2\,(\nabla^g_{fZ}F)(fX,Y)
 - (\nabla^g_{fX}F)(fY,Z) \\
 \notag
 && + (\nabla^g_{fY}F)(fX,Z)
 - (\nabla^g_{X}F)(Y,Z) + (\nabla^g_{Y}F)(X,Z)
   - (\nabla^g_{X}F)(fY,Z) \\
 && - (\nabla^g_{fY}F)(X,Z) - (\nabla^g_{Y}F)(fX,Z)
   - (\nabla^g_{fX}F)(Y,Z) \big\},
\end{eqnarray}
and the special Einstein connection is given by $\nabla=\nabla^g+\frac12\,T$, that is,
\begin{eqnarray*}
\notag
 && g(\nabla_XY,Z) = g(\nabla^g_XY,Z) \\
 &&\hskip19mm +\frac12\bigl\{(\nabla^g_{fZ}F)(fX,Y)+(\nabla^g_{fY}F)(fX,Z)-(\nabla^g_XF)(Y,Z)\bigr\}.
\end{eqnarray*}
\end{corollary}

One can find using \eqref{acm2} that the symmetric part of Einstein connection \eqref{genconein-acm} coincides with the Levi-Civita connection if and only if the following equality is true:
\begin{equation}\label{eq:E-spE}
(\nabla^g_XF)(Y,Z)+(\nabla^g_YF)(X,Z)=(\nabla^g_{fX}F)(fY,Z)+(\nabla^g_{fY}F)(fX,Z) .
\end{equation}
As a result, we obtain the following statement.

\begin{corollary}
An Einstein connection \eqref{genconein-acm} of an a.c.m. manifold is special if and only if the equality \eqref{eq:E-spE} holds for $X,Y,Z\in\ker\eta$.
\end{corollary}

\begin{example}
\rm
Take a $2n$-dimensional almost Hermitian manifold $(M',J,g')$, where $J^2=-{I}$ and a real line $(\RR, dt^2)$.
Consider the Riemannian product
$(M,g)=\bigl(M'\times\RR,\; g'\oplus dt^2\bigr)$. Then $(M^{2n+1},f,Q,\xi,\eta,g)$ is a weak a.c.m. manifold, with $\xi=\partial_t$, $\eta=dt$,
$f=\sqrt{\lambda}\,J$ and $f^2=-\lambda\,I$ on $\ker\eta$, and $Q=\lambda I + (1-\lambda)\,\eta\otimes\xi$, where $\lambda\in\RR_+$.
On $TM'$ we~get
\begin{eqnarray*}
\widetilde Q=(\lambda-1)\,{I}, \quad P={I}-f^2=(\lambda+1)\,{I},\quad P^{-1}=(\lambda+1)^{-1}\,{I}.
\end{eqnarray*}
The 2-form $F$ corresponding to $f$ is given on $TM'$ by $F(X,Y)=g(X,\sqrt\lambda J Y)$.
As a consequence,
 $(\nabla^g_XF)(Y,Z)=0$,
whenever $X$ belongs to one factor and at least one of $Y,Z$ belongs to the second factor.

If an Einstein connection $\nabla$ satisfies the
$Q$-$T$-condition \eqref{E-Q-T-condition}, then
 $T(X,Y)=0$ for $X\in TM',\, Y\perp TM'$.
Hence, the torsion tensor splits into the sum
 $T=T|_{TM'} \oplus T|_{T\RR}$.
Substituting $\widetilde Q|_{\,TM'} = (\lambda-1)\,{I}$ into \eqref{Eq-QT-solution}, then using
Theorem~\ref{T-w.c.m} and the equalities
\[
 (\nabla^{g}_X F)(J Y, J Z)= -\,(\nabla^{g}_X F)(Y,Z),
 \quad
 (\nabla^{g}_X F)(J Y, Z)= (\nabla^{g}_X F)(Y, J Z),
\]
we uniquely determine the torsion component $T\,|_{TM'}$ by
\begin{eqnarray}\label{E-sol-lambda}
\notag
 && \tfrac12(\lambda+1)^2\lambda\,T\,|_{\,TM'}(Y,Z,X) = -(\nabla^{g}_Y F)(Z,X) -(\nabla^{g}_Z F)(X,Y) \\
\notag
&& +\lambda\big\{2\,(\nabla^{g}_{JX} F)(Y,JZ) -(\nabla^{g}_{JY} F)(JZ,X) -(\nabla^{g}_{JZ} F)(X,JY) \big\} \\
&& +(\lambda-1)\big\{ 2\,(\nabla^{g}_X F)(Y,Z) -(\lambda+\tfrac12)\,dF(X,JY,JZ) -\tfrac12\,(\lambda+2)\,dF(X,Y,Z) \big\}.
\end{eqnarray}
If $(M',J,g')$ is a K\"ahler manifold, then $\nabla^{g}F=0$ and $dF=0$; hence, all torsion components vanish, and the Einstein connection
on $(M,g)$ reduces to the Levi-Civita connection.
The~formula \eqref{E-sol-lambda} for $\lambda=1$ gives the solution \eqref{Eq-2.13b}, which applies to a.c.m. manifolds.
\end{example}

\begin{remark}\rm
(i)~Chinea and Gonzalez \cite{CG-1990} (see also \cite{APDK-2026})
obtained a classification of a.c.m. manifolds based on the behaviour of the covariant derivative
$\nabla^{g} F$ and similar to the classification in \cite{gh-1980} of almost Hermitian manifolds,
and using Theorem~\ref{T-w.c.m}, we can represent explicitly the torsion of an Einstein connection for them.

(ii)~In \cite{rst-68}, we presented an Einstein connection (and a special Einstein connection)
for a subclass of a.c.m. manifolds satisfying the $f^2$-$T$-condition.
For an a.c.m. manifold, the $f^2$-$T$-condition is compatible with the splitting
$TM=\ker\eta\oplus \{\xi\}$.
Indeed, since $f^2=-I+\eta\otimes\xi$,
we have $f^2X=-X$
for any $X\in\ker\eta$, and the $f^2$-$T$-condition
\[
 T(f^2X,Y)=T(X,f^2Y)=f^2T(X,Y)
\]
implies
$T(X,Y)=-f^2T(X,Y)$.
Hence, the torsion $T(X,Y)$ has no $\xi$-component, i.e.,
\[
T(X,Y)\in\ker\eta
\ \Longleftrightarrow \ T(X,Y,\xi)=0
\qquad (X,Y\in\ker\eta) .
\]
Therefore, the subclass of a.c.m. manifolds satisfying the $f^2$-$T$-condition is charac\-terized by the fact that the torsion is adapted to the decomposition
 $TM=\ker\eta\oplus \{\xi\}$,
in particular, the torsion of any two horizontal vectors is again horizontal.

(iii)
For $Q=I$, the equation \eqref{Eq-QT-solution2} reduces to \eqref{Eq-2.13b}.
Indeed, for an a.c.m. manifold and for horizontal vector fields
\(X,Y,Z\in \ker\eta\), we have
\[
Q=I,\qquad \widetilde Q=Q-I=0,\qquad f^2=-I,\qquad f^4=I,
\qquad P=I-f^2=2I.
\]
Hence in \eqref{Eq-QT-solution2} all \(\widetilde Q\)-terms vanish, while
 $(I+\tfrac12\widetilde Q)^2f^4X=X$.
Thus, \eqref{Eq-QT-solution2} reduces to
\begin{eqnarray*}
\notag
 2\,T(Y,Z,X) =&&\
 (\nabla^g_X F)(fY,fZ) -(\nabla^g_Y F)(Z,X)
 -(\nabla^g_Z F)(X,Y) +(\nabla^g_{fX}F)(fY,Z) \\
&&+(\nabla^g_{fX}F)(Y,fZ) -(\nabla^g_{fY}F)(fZ,X)
 -(\nabla^g_{fZ}F)(X,fY) .
\end{eqnarray*}
Using the equalities on \(\ker\eta\):
$(\nabla^g_XF)(fY,fZ)=-(\nabla^g_XF)(Y,Z)$
and $(\nabla^g_{fX}F)(Y,fZ)=(\nabla^g_{fX}F)(fY,Z)$,
we obtain \eqref{Eq-2.13b}.
The same equation \cite[Eq.~3.8]{rst-68} was obtained under assumption of the $f^2$-$T$-condition.
Hence, Theorem~\ref{T-s-w.c.m} generalizes
\cite[Theorem~3.3]{rst-68}.
\end{remark}

\begin{remark}\rm
The following question arises:
How are the special Einstein connections presented in \cite[Theorem~3.7]{rst-68} and Theorem~\ref{T-s-w.c.m} related?
In \cite{rst-68},  the $f^2$-$T$-condition is used, which is adapted to the decomposition
$TM = \ker\eta \oplus \{\xi\}$.
This condition implies that the torsion is horizontal on $\ker\eta$, \ie,
$T(X,Y,\xi)=0 \ (X,Y \in\ker\eta)$.
This allows us to derive explicit formulas for the torsion of special Einstein connections.

In the present paper, we replace the $f^2$-$T$-condition by the $Q$-$T$-condition, which is naturally adapted to weak a.c.m. structures. Although this condition does not imply $T(X,Y,\xi)=0$,
it yields invariance of $dF$ and $\nabla^g F$ under the action of $Q$ (see Lemma~\ref{L-QT-df-nablaF}),
which allows one to determine the torsion and obtain the general formula (Theorem~\ref{T-w.c.m}).

In the classical a.c.m. case we have $Q=I$, and the $Q$-$T$-condition becomes trivial. Then the
formula \eqref{Eq-QT-solution2} reduces to \eqref{Eq-2.13b}. Moreover, for special Einstein connections, the condition
\(
 K_{X}Y = -K_{Y}X
\)
yields the formula \eqref{Eq-2.13-sp}, which coincides with the expression obtained in \cite[Theorem 3.7]{rst-68}.
Therefore, special Einstein connections from \cite{rst-68} appear as a particular case of those given in Theorem~\ref{T-s-w.c.m}, corresponding to $Q=I$. Thus, Theorem~\ref{T-s-w.c.m} generalizes the result from the $f^2$-$T$-condition to a more general~case.
\end{remark}

\begin{example}\rm
 Let $(M^{2n+1},f,\xi,\eta,g)$  be a Sasakian manifold. For this manifold it holds
\begin{eqnarray}\label{E-sasaki}
 (\nabla^g_X{f})Y=g(X,Y)\,\xi -\eta(Y)X ,
\end{eqnarray}
\eg, \cite{Blair-survey}.
Also note that ``weak Sasakian manifolds" are Sasakian manifolds, that is $\widetilde Q=0$, see \cite{rov-survey24}.
Since $F=d\eta$ for a contact metric manifold, we have $dF=0$.
Setting $Y=\xi$ in \eqref{E-sasaki} gives
$\nabla^g_X\,\xi=-fX$. Hence, $\nabla^g_\xi\,f=0$, $\nabla^g_\xi\,\xi=0$.
Using \eqref{E-sasaki} and \eqref{E-nabla-F-f}, gives
\begin{equation}\label{Eq-4.01}
 (\nabla^{g}_X F)(Y,\xi)
  = \eta(X)\eta(Y) - g(X,Y)= -g(fX,fY).
\end{equation}
Let $\nabla\ne\nabla^g$ be an Einstein connection with torsion $T$ on $M$.
Using \eqref{Eq-4.01} in \eqref{E-xi-YX} and \eqref{E-x-y-xi} with $\widetilde Q=0$ and applying
$f(I+f)^{-1}=\frac12\,(I-f)f$ and $f\xi=0$, gives
\begin{eqnarray*}
\notag
 T(\xi,Y,X) &&= 2\,g(fX, f(I+f)^{-1}Y)
= g(fX,(f-f^2)Y) \\
&&= g(fX,Y+fY) = F((I+f)Y,X) ,\\
 \notag
 T(X,Y,\xi) &&
 =-2\,F( ( I+f(I+f)^{-1} )X, ( I+f(I+f)^{-1} )Y) \\
 && =- 2\,F( ( I+\tfrac12 (I-f)f )X,
 ( I+\tfrac12 (I-f)f )Y ) =- 2\,F(X,Y).
\end{eqnarray*}
Taking $X=\xi$ or $Y=\xi$, the above equations verify \eqref{E-w.c.m-xi}.
For a Sasakian manifold we get
\[
(\nabla^g_XF)(Y,Z)=g\bigl(Y,(\nabla^g_Xf)Z\bigr)
=g(X,Z)\eta(Y)-\eta(Z)g(X,Y).
\]
Thus
\(
(\nabla^g_XF)(Y,Z)=0\ (X,Y,Z\in\ker\eta)
\).
Using this in \eqref{Eq-2.13b}, gives
 $2\,T(Y,Z,X)=0$ for $X,Y,Z\in\ker\eta$.

We check this in the Einstein metricity condition \eqref{metein0}. Since on the distribution
$\ker\eta$ we also have
$(\nabla^g_X g)(Y,Z)=0$
and
$(\nabla^g_XF)(Y,Z)=0$,
it follows that
$(\nabla^g_X G)(Y,Z)=0\ (X,Y,Z\in\ker\eta)$.
We have $T(X,Y,Z)=0$ for $X,Y,Z\in\ker\eta$.
By \eqref{E-tordfnew},
this gives
$K(X,Y,Z)=0$,
thus, $K_XY$ is proportional to $\xi$. Therefore, for $Z\in\ker\eta$,
$$
 G(K_XY,Z)=0,
\qquad
 G(Y,K_XZ)=0.
$$
By the above,
$$
(\nabla_XG)(Y,Z)
=
(\nabla^g_XG)(Y,Z)=0,
$$
\eqref{metein0} is satisfied
on $\ker\eta$.
%
Take $X,Y,Z\in\ker\eta$, then
$\eta(X)=\eta(Y)=\eta(Z)=0$
and
$f^2|_{\ker\eta}=-I|_{\ker\eta}$.
Hence, on $\ker\eta$ the structure reduces to an almost Hermitian one.
Using \eqref{acm2}, the formula \eqref{Eq-2.13b} coincides with the almost
Hermitian torsion formula \eqref{Eq-2.13}.
Thus, substituting \eqref{Eq-2.13b} into the
Einstein metricity condition \eqref{metein0},
gives an identity.
Thus, \eqref{metein0} is satisfied for $X,Y,Z\in\ker\eta$.
\end{example}

\FloatBarrier

 \bigskip

{\small\rm\baselineskip=10pt
 \baselineskip=10pt
 \qquad Prof. Vladimir Rovenski\par
 \qquad Faculty of Natural Science\par
 \qquad Department of Mathematics\par
 \qquad University of Haifa\par
 \qquad 3498838 Haifa, Israel\par
 \qquad {\tt vrovenski@univ.haifa.ac.il}

 \bigskip \smallskip

 \qquad Prof. Milan Zlatanovi\'c\par
 \qquad Faculty of Science and Mathematics\par
 \qquad Department of Mathematics\par
 \qquad University of Ni\v{s}\par
 \qquad 18000 Ni\v s, Serbia\par
 \qquad {\tt zlatmilan@yahoo.com}

  \bigskip \smallskip

 \qquad Miroslav D. Maksimovi\'c\par
 \qquad Faculty of Sciences and Mathematics\par
 \qquad Department of Mathematics\par
 \qquad University of Pri\v{s}tina\par
 \qquad Kosovska Mitrovica, Serbia\par
 \qquad {\tt miroslav.maksimovic@pr.ac.rs}
 }

 \end{document}